\documentclass[letterpaper, 10 pt, conference]{ieeeconf}  

\IEEEoverridecommandlockouts                              

\usepackage{graphicx}
\usepackage{colortbl}
\usepackage{amsmath}
\usepackage{amssymb}
\usepackage{subcaption}
\usepackage{acronym}
\usepackage[noadjust]{cite}	
\usepackage[capitalise]{cleveref}

\usepackage{pgfplots}
\usepackage{tikz}
\usetikzlibrary{calc,arrows}
\usepackage{tikzscale}
\newlength{\axiswidth} 
\newlength{\axisheight} 
\newtheorem{defn}{Definition}
\newtheorem{thm}{Theorem}
\newtheorem{assum}{Assumption}
\newtheorem{lem}{Lemma}
\newtheorem{prop}{Proposition}
\newtheorem{rem}{Remark}
\newtheorem{cor}{Corollary}
\newtheorem{prob}{Problem}

\crefname{assum}{Assumption}{Assumptions}
\crefname{defn}{Definition}{Definitions}

\newcommand{\R}{\mathbb{R}}
\newcommand{\N}{\mathbb{N}}

\newcommand{\agents}{\mathcal{A}}
\newcommand{\tasks}{\mathcal{T}}
\newcommand{\BAop}{\mathcal{B}}
\newcommand{\Bop}{\operatorname{B}}
\newcommand{\LexAop}{\mathcal{S}}

\newcommand{\graph}{\mathcal{G}}

\newcommand{\vertex}{\mathcal{V}}
\newcommand{\edges}{\mathcal{E}}
\newcommand{\weights}{\mathcal{W}}
\newcommand{\subedges}{\bar{\edges}}
\newcommand{\subgraph}{\bar{\graph}}
\newcommand{\subagents}{\bar{\agents}}
\newcommand{\subtasks}{\bar{\tasks}}
\newcommand{\vass}{\mathcal{P}}
\newcommand{\bin}{\mathbb{B}}
\newcommand{\bott}{\operatorname{b}}
\newcommand{\ebottmm}{\operatorname{e}}
\newcommand{\allebott}{\operatorname{E}}

\newcommand{\robOp}{\operatorname{r}}

\newcommand{\dt}{b}
\newcommand{\ds}{a}
\newcommand{\dist}{\operatorname{d}}
\newcommand{\locCon}{\mathcal{L}}
\newcommand{\bigO}{\mathcal{O}}
\newcommand{\vref}{v^{\textrm{ref}}}

\definecolor{a1}{rgb}{0.8929,0.5951,0.4381}
\definecolor{a2}{rgb}{0.3320,0.5298,0.7359}
\definecolor{a3}{rgb}{0.8212,0.4188,0.5180}
\definecolor{a4}{rgb}{0.0417,0.3354,0.5789}
\definecolor{a5}{rgb}{0.1077,0.6225,0.6454}

\acrodef{lap}[LSAP]{Linear Sum Assignment Problem}
\acrodef{bap}[BAP]{Bottleneck Assignment Problem}
\acrodef{lex}[LexBAP]{Lexicographic Bottleneck Assignment Problem}

\title{\LARGE \bf
Collision Avoidance Based on Robust Lexicographic Task Assignment
}

\author{Tony A.\ Wood, Mitchell Khoo, Elad Michael, Chris Manzie, and Iman Shames
\thanks{This work was supported by Defence Science and Technology Group through research agreements MyIP:7558, MyIP:7562, and MyIP:9156.}
\thanks{The authors are with the Control and Signal Processing Group, Department of Electrical and
Electronic Engineering, University of Melbourne, Parkville, VIC 3010, Australia. Email: {\tt\small  \{wood.t, manziec, iman.shames\}@unimelb.edu.au} and~{\tt\small  \{khoom1, eladm\}@student.unimelb.edu.au}.}%
}

\begin{document}

\maketitle
\thispagestyle{empty}
\pagestyle{empty}

\begin{abstract}
Traditional task assignment approaches for multi-agent motion control do not take the possibility of collisions into account. This can lead to challenging requirements for path planning. We derive an assignment method that not only minimises the largest distance between an agent and its assigned destination but also provides local constraints for guaranteed collision avoidance. To this end, we introduce a sequential bottleneck optimisation problem and define a notion of robustness of an optimising assignment to changes of individual assignment costs. Conditioned on a sufficient level of robustness in relation to the size of the agents, we construct time-varying position bounds for every individual agent. These local constraints are a direct byproduct of the assignment procedure and only depend on the initial agent positions, the destinations that are to be visited, and a timing parameter. We prove that no agent that is assigned to move to one of the target locations collides with any other agent if all agents satisfy their local position constraints. We demonstrate the method in a illustrative case study. 
\end{abstract}


\section{Introduction}
For  
autonomous systems with multiple agents and multiple tasks the first decision to be made is how to allocate the tasks to the agents. This type of decision is referred to as an assignment problem and may involve any of a variety of objectives. In \cite{Chopra2017ToR,Choi2009ToR} the sum of individual costs incurred for assigning robotic agents to tasks is minimised in a so called \ac{lap}. 
An overview of assignment problems is given in~\cite{Pentico2007EJoOR,Gerkey2004IJoRR,Burkard2012Book}. 

We focus here on a specific type of assignment problem where the largest incurred individual agent-to-task assignment cost 
is minimised. This objective is referred to as the \ac{bap}. It commonly applies to multi-agent problems where tasks are completed in parallel and the overall completion time is of interest, as in 
\cite{Shames2017CDC} for instance. A \ac{lex} describes a subclass of the \ac{bap} where not only the largest assignment cost but also all other assignment costs are minimised with a sequence of decreasing hierarchy~\cite{Burkard1991ORL}. A review of the state-of-the-art methods to solve the \ac{bap} is provided in~\cite{
Burkard2012Book}. 
In~\cite{Khoo2019CDC} an algorithm to solve the \ac{bap} with distributed computation is introduced. Global sensitivity of the bottleneck optimising assignment with respect to changes in the assignment costs is investigated in~\cite{Michael2019ECC}.  

In 
applications where the tasks involve agents moving towards desired locations, collisions between agents may occur. Motion control with inter-agent collision avoidance is a heavily researched problem and many different strategies exist, e.g., reciprocal collision avoidance~\cite{vandenBerg2011RR,
AlonsoMora2018ToR}, barrier certificates~\cite{Wang2017ToR}, buffered Voronoi cells~\cite{Zhou2017RAL}, and  multi-agent rapidly-exploring random trees~\cite{Desaraju2012AR}.  
In \cite{Turpin2014AR} independent trajectories for all possible agent-destination pairs are designed, tasks are 
allocated by selecting an assignment that minimises the lexicographic cost of the chosen trajectories, and collision are avoided by subsequently delaying the starting times of certain agents. 
Collision avoidance awareness criteria are included in a task assignment problem in 
\cite{Wu2019ISMRMAS} by designing a quadratic assignment cost function that incentivises collision-free paths. 

Similar to \cite{Shames2011IJoRaNC,MacAlpine2015CoAI,Turpin2014IJRR}, we investigate intrinsic properties of assignment problems that provide conditions for which inter-agent collisions are avoided.  In \cite{Shames2011IJoRaNC} it was first shown that straight lines connecting agents to allocated destinations do not intersect if the allocation corresponds to the solution of an \ac{lap}, with Euclidean distances between initial agent positions and target destinations as assignment costs. Conditioned on sufficiently large distances between initial positions and target locations, a time parametrisation of straight-line trajectories is derived in~\cite{Turpin2014IJRR} such that agents with finite extent do not collide when following an assignment that solves an \ac{lap}, with squared distances as assignment costs. We consider a case where the largest distance between any agent and its assigned destination is to be minimised. The non-intersection properties of the \ac{lap} 
do not hold for the \ac{bap}. 
However, assuming agents are point-masses that travel with constant velocity on straight paths towards their allocated destinations, it is shown in \cite{MacAlpine2015CoAI} that no collisions occur if the task allocation is based on either a \ac{lex}, with distances as assignment costs, 
or a designated assignment problem which combines a \ac{bap}, with distances as costs, and an \ac{lap}, with squared distances as costs.  
 We assume that agents have finite extent, consider arbitrary distance metrics, and do not limit the trajectories to straight lines.

The main contribution of this paper is the derivation of local position constraints for every agent. 
Rather than determining trajectories for all agents, the method proposed here provides time-varying sets of positions that guarantee collision avoidance but leave some degree of freedom for low-level path planning and motion control.  In particular, we propose a sequential bottleneck task assignment approach  
and show that it produces the unique optimising solution of a \ac{lex} if 
robustness conditions are satisfied. We construct the local constraints that prevent collisions by quantifying the robustness to changes of the distances to the destinations. 
The constraints for an individual agent do not explicitly depend on the positions of the other agents. Information on other agents only gets accounted for via the assignment costs. 
We prove that it is sufficient for every agent to satisfy its 
local constraints in order to guarantee that no agent that is assigned to one of the tasks will collide with any other agent. %



The paper is structured as follows. We 
formulate the assignment and collision avoidance objectives in Section~\ref{sec:ProblemFormulation}. In Section~\ref{sec:LexAss} we introduce the 
assignment procedure and investigate its properties. These properties are applied to derive sufficient conditions for collision avoidance in Section~\ref{sec:CollisionAvoidance}. We illustrate the resulting agent position constraints in a case study in Section~\ref{sec:CaseStudy} before concluding in Section~\ref{sec:conclusion}.  

\section{Problem Formulation}\label{sec:ProblemFormulation}
We consider $m$ agents, $\agents:=\{\alpha_l\}^m_{l=1}$, with initial positions, $p_i(0)\in\R^{n_p}$, $n_p\in\N$, $i\in\agents$, and $n$ target destinations, $g_j\in\R^{n_p}$, $j\in \tasks$, where $\tasks:=\{\tau_l\}^n_{l=1}$ and $\tau_l$ represents the task of going to destination $g_l$. Without loss of generality we assume there are at least as many agents as tasks, $m\geq n$, and consider only non-trivial cases where there are multiple agents, $m > 1$, and at least one task, $n\geq 1$. 

We define a set of binary variables, $\Pi:=\{\pi_{i,j}\in\{0,1\}\,|\,(i,j)\in\agents\times\tasks\}$, that indicates agent $i\in\agents$ being \emph{assigned} to task $j\in\tasks$ if and only if $\pi_{i,j}=1$.  We call such a set an \emph{assignment} and denote the set of all assignments by $\bin_{\agents,\tasks}$. We assume that a cost is incurred when an agent proceeds to visit an assigned location.
\begin{assum} \label{ass:dist} 
The cost of assigning an agent, $i\in\agents$, to complete a task, $j\in\tasks$, is given by the distance of the initial agent position to the corresponding destination, $\dist(p_i(0),g_j)$,
where $\dist:\R^{n_p}\times\R^{n_p}\rightarrow [0,\infty)$ is an arbitrary distance function that 
satisfies the triangle inequality,
\begin{align}\label{eq:triangleIneq}
  \dist(p,p')\leq \dist(p,p'') + \dist(p'',p')\,,
\end{align} for all $p,p',p''\in\R^{n_p}$.
\end{assum}

At any given time, $t$, the centroid location of agent $i\in\agents$ is given by a point, $p_i(t)$. However, the body of the agent occupies a finite volume in $\R^{n_p}$. If two agents lie in close proximity, they may collide. 
\begin{defn}[Inter-agent safety distance]
\label{defn:safetydistance}
  Agents $i,i'\in\agents$, $i\neq i'$, do not collide with each other at time $t$ if
  \begin{align}\label{eq:CollisionAvoidanceConstraint}
    \dist(p_i(t),p_{i'}(t))> s_{i,i'}\,,
  \end{align} where $s_{i,i'}=s_{i',i}\geq 0$ is known and named the \emph{safety distance} between $i$ and $i'$.
\end{defn}
We note that the collision avoidance condition in~\eqref{eq:CollisionAvoidanceConstraint} couples the motion control problem of agents $i$ and $i'$. Satisfying this condition introduces a non-convex constraint with respect to the positions of the agents, $p_i(t)$ and $p_{i'}(t)$, and is therefore challenging. 
 The objective considered in this paper is to derive local position constraints for each individual agent, that sufficiently guarantee avoidance of collisions of agents that are assigned to tasks such that the maximum agent-to-destination distance, the so-called \emph{bottleneck}, is minimised. 
\begin{prob}\label{prob:main} 
  Suppose that one agent is assigned to every target destination such that the largest individual cost is minimised according to the \ac{bap}, 
\begin{subequations}\label{eq:BAP}
  \begin{align}
    \min_{\Pi\in\bin_{\agents,\tasks}} &&&\max_{(i,j)\in\agents\times\tasks}\pi_{i,j}\dist(p_i(0),g_j)\label{eq:bapobjective}\\
    \textrm{s.t.}&&&\sum_{i\in\agents}\pi_{i,j} = 1 &\forall j&\in\tasks\,,\\
    &&&\sum_{j\in\tasks}\pi_{i,j} \leq 1 &\forall i&\in\agents\,.   
  \end{align}
\end{subequations}
Find time-varying sets of safe positions, $\locCon_i(t)\subset\R^{n_p}$, for all agents, $i\in\agents$, such that no agent assigned to a task collides with any other agent over a time interval, $t\in[0,T]$, i.e., 
if $p_i(t)\in\locCon_i(t)$ for all $i\in\agents$, $t\in[0,T]$, then $\dist(p_{i^*}(t),p_{i'}(t))> s_{i^*,i'}$,
for all $i^*\in\{i\in\agents\,|\,\sum_{j\in\tasks}\pi_{i,j}=1\}$, $i'\in\agents$, 
$t\in[0,T]$. 
\end{prob}

\section{Task Assignment}\label{sec:LexAss}
In this section we introduce task assignment criteria and derive properties that will be applied to obtain collision avoidance conditions in Section~\ref{sec:CollisionAvoidance}. We illustrate details of the assignment procedure 
with an 
example 
in Section~\ref{sec:exass}.

Given a set of agents, $\agents$, and a set of tasks, $\tasks$, where $|\agents|=m>1$ and $m\geq|\tasks|=n\geq 1$, 
we define the complete bipartite \emph{assignment graph} $\graph:=(\agents,\tasks,\edges)$, with vertex set, $\vertex:=\agents\cup \tasks$, and edge set, $\edges:=\agents\times \tasks$. 
We also define the set of \emph{assignment weights}, 
$\weights:=\{w_{i,j}\geq 0\,|\, (i,j)\in\edges\}$. An assignment, $\Pi\in\bin_{\agents,\tasks}$, is an \emph{admissible} allocation of a subset of tasks, $\subtasks\subseteq\tasks$, to a subset of agents, $\subagents\subseteq\agents$, with respect to edge subset, $\hat{\edges}\subseteq\subedges:=\subagents\times\subtasks$, if all considered tasks in $\subtasks$, are assigned to one agent in $\subagents$ and all these agents are assigned to at most one task. The set of admissible assignments for subgraph $\hat{\graph}:=(\subagents,\subtasks,\hat{\edges})$ is given by, 
\begin{align*}
  \vass_{\subagents,\subtasks}(\hat{\edges}):= \Big\{\Pi\!\in\! \bin_{\agents,\tasks}\,\Big|\,&\forall j\!\in\! \subtasks\, \sum_{i\in \{i'\in\subagents|(i',j)\in\hat{\edges}\}}\pi_{i,j} = 1,\\
&\forall i\!\in\! \subagents\,\sum_{j\in \{j'\in\subtasks|(i,j')\in\hat{\edges}\}}\pi_{i,j} \leq 1\Big\}.
\end{align*}

\subsection{Bottleneck Assignment}\label{sec:bottass}
The \ac{bap} formulated in~\eqref{eq:BAP} can be solved efficiently in polynomial time, see~\cite{Burkard2012Book}. Independent of how the \ac{bap} is solved, we define operators associated to it in the following. 
\begin{defn}[Bottleneck assignment]\label{def:bottass}
  We consider a subgraph of the assignment graph, $\hat{\graph
  }=(\subagents,\subtasks,\hat{\edges})$, the assignment weights, 
$\weights$, and an assignment, 
$\Pi\in\bin_{\agents,\tasks}$. 
We define a function that returns the largest value among the weights corresponding to assigned edges in $\hat{\edges}$, 
\begin{align}\label{eq:maxweight}
  \bott(\Pi,\hat{\edges},\weights) &:= \max_{(i,j)\in\hat{\edges}} \:\pi_{i,j}w_{i,j}\,, 
\end{align} 
a function that returns the \emph{bottleneck weight} of subgraph~$\hat{\graph}$,
\begin{align}\label{eq:bottleneckweight}
  \Bop_{\subagents,\subtasks}(\hat{\edges},\weights)&:= \min_{\Pi\in \vass_{\subagents,\subtasks}(\hat{\edges})} \bott(\Pi,\hat{\edges},\weights)\,,
\end{align} a map that returns the 
\emph{bottleneck 
minimising assignments}, 
\begin{align*}
  \BAop_{\subagents,\subtasks}(\hat{\edges},\weights)&:= \arg\min_{\Pi\in \vass_{\subagents,\subtasks}(\hat{\edges})} \bott(\Pi,\hat{\edges},\weights)\,,
\end{align*} 
and the set of edges with weight equal to the bottleneck 
\begin{align}\label{eq:allbottedges}
    \allebott_{\subagents,\subtasks}(\hat{\edges},\weights):=\big\{(i,j)\in\hat{\edges}\,\big|\,w_{i,j}=\Bop_{\subagents,\subtasks}(\hat{\edges},\weights)\big\}\,.
\end{align}
\end{defn}

Next, we quantify how sensitive an optimising assignment is to changes of the bottleneck weight. Specifically, we define the robustness margin 
that captures the difference between the optimal bottleneck cost and the cost of the next best assignment when a specific bottleneck edge is removed. In \cref{sec:CollisionAvoidance} we use this information to determine by how much the positions of agents can vary such that the optimality of the assignment is maintained and collisions do not occur.  
\begin{defn}[Robustness margin]\label{def:criticaledge}
  Given the assignment weights, $\weights$, we consider the complete bipartite graph formed by a subset of agents, $\subagents$, and a subset of tasks, $\subtasks$, i.e. subgraph $\subgraph:=(\subagents,\subtasks,\subedges)$, with edge set $\subedges=\subagents\times\subtasks$. For $|\subedges|>1$, we define the set of so-called 
  \emph{maximum-margin bottleneck edges} as,
  \begin{align*}
        \ebottmm_{\subagents,\subtasks}(\weights):= \arg\max_{(i,j)\in\allebott_{\subagents,\subtasks}(\subedges,\weights)}\Bop_{\subagents,\subtasks}(\subedges\setminus\{(i,j)\},\weights)\,,
  \end{align*} and the corresponding \emph{robustness margin},
  \begin{align*}
      \robOp_{\subagents,\subtasks}(\weights) &:= \max_{(i,j)\in\allebott_{\subagents,\subtasks}(\subedges,\weights)}\Bop_{\subagents,\subtasks}(\subedges\setminus\{(i,j)\},\weights) - w_{i,j}\,.
  \end{align*}
  For $|\subagents|=|\subtasks|=|\subedges|=1$, the maximum-margin bottleneck edge is set to be the singleton edge $\ebottmm_{\subagents,\subtasks}(\weights)=\subedges$ and the robustness margin is assumed to be infinity, $\robOp_{\subagents,\subtasks}(\weights)=\infty$.
%
\end{defn}

The following proposition links the robustness margin to uniqueness properties of bottleneck minimising assignments. While a strictly positive robustness margin does not imply that there is a unique bottleneck minimising assignment, it does ensure that all bottleneck minimising assignments have in common that they assign all agent-task pairs in the set of maximum-margin bottleneck edges. 
\begin{prop}[Proof in Appendix~\ref{app:ProofUniqueMMB}]\label{prop:UniqueMMB}
   Given the assignment weights, $\weights$, we consider the complete bipartite graph, $\subgraph=(\subagents,\subtasks,\subedges)$, formed by a subset of agents, $\subagents\in\agents$, and a subset of tasks, $\subtasks\in\tasks$, with edge set $\subedges=\subagents\times\subtasks$, where $|\subedges|>1$. If the robustness margin is strictly positive, $\robOp_{\subagents,\subtasks}(\weights)>0$, then all bottleneck minimising assignments for subgraph $\subgraph$ assign all maximum-margin bottleneck edges, i.e., $\pi_{i^*,j^*} = 1$,  $\forall (i^*,j^*)\in\ebottmm_{\subagents,\subtasks}(\weights)$, $\forall\Pi=\{\pi_{i,j}\in\{0,1\}|(i,j)\in\edges\}\in\BAop_{\subagents,\subtasks}(\subedges,\weights)$.
\end{prop}

\subsection{Sequential Bottleneck Assignment}\label{sec:seqbottass}
The \ac{bap} 
does not fully determine all agent-task pairings.  
Given the primary assignment criterion introduced in~\eqref{eq:BAP}, we define a sequence of criteria of decreasing hierarchy. 
\begin{defn}[Sequential bottleneck assignment]\label{def:lexass} We consider the agents, $\agents$, the tasks, $\tasks$, and the assignment weights, $\weights$. An assignment, $\Pi^*$, is \emph{sequential bottleneck optimising} if it is bottleneck minimising for the assignment graph, $\graph=(\agents,\tasks,\edges)$, and the sequence of subgraphs $\subgraph_2,\subgraph_3,\dots,\subgraph_n$, where $\subgraph_k=(\subagents_k,\subtasks_k,\subedges_k)$ is the complete bipartite graph of the subset of agents, $\subagents_k\subset\agents$, 
and the subset of tasks, $\subtasks_k\subset\tasks$, 
obtained by removing a maximum-margin bottleneck agent and task from $\subgraph_{k-1}$, i.e.,
$\Pi^* \in \LexAop_{\agents,\tasks}(\weights)$,     
\begin{align*}
  \LexAop_{\agents,\tasks}(\weights)\!:=\!\big\{\!\Pi\!\in\!\bin_{\agents,\tasks}\big|\forall k\!\in\!\{1,\dots,n\}\, \Pi\!\in\!\BAop_{\subagents_k,\subtasks_k}(\subedges_k,\weights)\!\big\}\!,
\end{align*} 
where $\subedges_k=\subagents_k\times\subtasks_k$, $\subagents_1=\agents$, $\subtasks_1=\tasks$, 
\begin{subequations}\label{eq:subgraph_k}
  \begin{align}
    \subagents_k&=\agents\setminus\{i^*_1,\dots,i^*_{k-1}\}\,,\\ \subtasks_k&=\tasks\setminus\{j^*_1,\dots,j^*_{k-1}\}\,,
  \end{align}
\end{subequations} 
with so-called \emph{$k$-th order bottleneck edge},
\begin{align}\label{eq:kthbott}
  (i^*_k,j^*_k)\in\ebottmm_{\subagents_k,\subtasks_k}(\weights)\,,
\end{align} and \emph{$k$-th order robustness margin},
\begin{align}\label{eq:kthmargin}
    \mu_k = \robOp_{\subagents_k,\subtasks_k}(\weights)\,.
\end{align} 
\end{defn}

Given a sequential bottleneck assignment, $\Pi^*\in\LexAop_{\agents,\tasks}(\weights)$, the set of all assigned agents is defined by the set of $k$-th order bottleneck agents of all orders, $\{i^*_1,\dots,i^*_n\}\subseteq\agents$. For the case where there are more agents than tasks, $m>n$, the set of unassigned agents is 
$\agents\setminus\{i^*_1,\dots,i^*_n\}$. The following two propositions show that the assignment of 
the $k$-th order bottleneck edge 
is invariant to an additive weight perturbation that is smaller than the $k$-th order robustness margin. We use this property to link the robustness margins to bounds on the deviation of the agent positions that guarantee collision avoidance in Section~\ref{sec:CollisionAvoidance}. 

\begin{prop}[Proof in Appendix~\ref{app:ProofPropAssEdgeSwitch}]\label{prop:AssEdgeSwitch} Given the agents, $\agents$, with $|\agents|=m>1$, the tasks, $\tasks$, with $m\geq |\tasks|=n\geq1$, and the assignment weights, $\weights=\{w_{i,j}\geq 0\,|\, (i,j)\in\agents\times\tasks\}$, we consider a sequential bottleneck optimising assignment, $\Pi^*\in\LexAop_{\subagents,\subtasks}(\weights)$, with $k$-th order bottleneck edges, $(i^*_k,j^*_k)$, 
and 
robustness margins, $\mu_k$, 
for $k\in\{1,\dots,n\}$, introduced in~\cref{def:lexass}. 
We have
  \begin{align*}
    w_{i^*_a,j^*_a} + \mu_a \leq \max\{w_{i^*_a,j^*_b},w_{i^*_b,j^*_a}\}\,,
  \end{align*} for all $a\in\{1,\dots,n-1\}$ and $b\in\{a+1,\dots,n\}$.
\end{prop}

\begin{prop}[Proof in Appendix~\ref{app:ProofPropUnassAgentSwitch}]\label{prop:UnassAgentSwitch}
Given the agents, $\agents$, with $|\agents|=m>1$, the strictly less tasks, $\tasks$, with $m>|\tasks|=n\geq 1$, and the assignment weights, $\weights=\{w_{i,j}\geq 0\,|\, (i,j)\in\agents\times\tasks\}$,  we consider a sequential bottleneck optimising assignment, $\Pi^*\in\LexAop_{\subagents,\subtasks}(\weights)$, with $k$-th order bottleneck edges, $(i^*_k,j^*_k)$, 
and 
robustness margins, $\mu_k$, 
for $k\in\{1,\dots,n\}$, introduced in~\cref{def:lexass}.
We have
  \begin{align*}
    w_{i^*_a,j^*_a} + \mu_a\leq w_{i',j^*_a}\,,
  \end{align*} for all $a\in\{1,\dots,n\}$ and $i'\in\agents\setminus\{i^*_1,\dots,i^*_n\}$.
\end{prop}

We note that the $k$-th order bottleneck edge defined in~\eqref{eq:kthbott}  may not be unique. 
However, if the robustness margin is strictly positive for all orders, it follows from Proposition~\ref{prop:UniqueMMB} that all possible $k$-th order bottleneck edges lead to a unique sequentially bottleneck optimising assignment.
\begin{cor}
  We consider a sequential bottleneck optimising assignment, $\Pi^*\in\LexAop_{\agents,\tasks}(\weights)$. If the $k$-th order robustness margin is strictly positive, $\robOp_{\subagents_k,\subtasks_k}(\weights)>0$, for all 
  $k\in\{1,\dots,n\}$, then the sequence of bottleneck weight values, $(w_{i^*_1,j^*_1},\dots,w_{i^*_n,j^*_n})$, is unique and lexicographically optimal, i.e., for all other assignments, $\Pi'\in\vass_{\agents,\tasks}(\edges)\setminus\{\Pi^*\}$, either $w_{i^*_1,j^*_1}<w_{i'_1,j'_1}$ or
  there exists an order, $l\in\{2,\dots,n\}$, such that $w_{i^*_k,j^*_k}=w_{i'_k,j'_k}$ for $k\in\{1,\dots,l-1\}$ and $w_{i^*_l,j^*_l}< w_{i'_l,j'_l}$, 
  where $(w_{i'_1,j'_1},\dots,w_{i'_n,j'_n})$ is a sequence of weights associated to assigned edges according to $\Pi'$ that are ordered with non-increasing values.
\end{cor}

We call a sequential bottleneck optimising assignment with strictly positive robustness margins for all orders a \emph{robust lexicographic assignment}. 
Such an 
assignment is a special case solution of a \ac{lex} with a 
unique optimiser. 
Finding the $k$-th order bottleneck and the corresponding robustness margin requires solving two \ac{bap}s that can each be computed with a complexity of $\bigO(|\subedges_k||\subtasks_k|)$, see~\cite{Khoo2019CDC}. From~\eqref{eq:subgraph_k} it follows that $|\subedges_k|=(m-k+1)(n-k+1)$ and $|\subtasks_k|=n-k$. Finding the $k$-th order bottleneck edges and robustness margins for all 
$k\in\{1,\dots,n\}$ can therefore be achieved with an overall complexity of $\bigO(mn^{3})$. 
\begin{rem}\label{rem:DistLex}
     The algorithm introduced in~\cite{Khoo2019CDC} solves the \ac{bap} without centralised decision making. Agents only have knowledge of weights associated to them, i.e., agent $i\in\agents$ has access to the subset of weights $\weights_i:=\{w_{i,j}\,|\,j\in\tasks\}$, and communicates local estimates of maximal and minimal edge weights to other agents. Because solving the sequential bottleneck assignment consists of solving $2n$ nested \ac{bap}s, it follows that a sequential bottleneck optimising assignment and the corresponding robustness margins can also be obtained with distributed computation. 
\end{rem}

\subsection{Assignment Example}\label{sec:exass}
\begin{figure}
    \begin{subfigure}{0.32\linewidth}
        \centering
         \tikzstyle{agent} = [draw,circle,fill=blue!20,minimum width=0.6cm]
\tikzstyle{task} = [draw,circle,fill=green!20,minimum width=0.6cm]
\tikzstyle{ass1} = [densely dotted,very thick]
\tikzstyle{ass2} = [densely dashed,thick]
\tikzstyle{bottleneck} = [red]
\tikzstyle{critical} = [orange]

\begin{tikzpicture}[scale=2.3]
	{\scriptsize
	\node[agent] (a1) {$\alpha_1$};
	\node[agent] (a2) at ($(a1) + (0,-0.5)$) {$\alpha_2$};
	\node[agent] (a3) at ($(a2) + (0,-0.5)$) {$\alpha_3$};
	\node[agent] (a4) at ($(a3) + (0,-0.5)$) {$\alpha_4$};
	
	\node[task] (t1) at ($(a1) + (0.75,-0.25)$){$\tau_1$};
	\node[task] (t2) at ($(a2) + (0.75,-0.25)$) {$\tau_2$};
	\node[task] (t3) at ($(a3) + (0.75,-0.25)$) {$\tau_3$};

	\draw[ass1] (a1) -- node [above] {4} (t1);
	\draw[] (a1) -- node [above,near start] {6} (t2);
	\draw[ass2] (a1) -- node [below,very near start] {2} (t3);
	
	\draw[] (a2) -- node [above,near end] {8} (t1);
	\draw[densely dashdotted, very thick, bottleneck] (a2) -- node [above,near start] {4} (t2);
	\draw[] (a2) -- node [below,near end] {9} (t3);
	
	\draw[critical] (a3) -- node [above,near end] {7} (t1);
	\draw[] (a3) -- node [below,near start] {9} (t2);
	\draw[] (a3) -- node [below,near end] {8} (t3);
	
	\draw[ass2] (a4) -- node [above,very near start] {2} (t1);
	\draw[] (a4) -- node [below,near start] {2} (t2);
	\draw[ass1] (a4) -- node [below] {1} (t3);
	}
\end{tikzpicture}
        \caption{Order $k=1$} 
        \label{fig:LexExample_k1}
    \end{subfigure}
    \hfill
    \begin{subfigure}{0.32\linewidth}
        \centering
         \tikzstyle{agent} = [draw,circle,fill=blue!20,minimum width=0.6cm]
\tikzstyle{task} = [draw,circle,fill=green!20,minimum width=0.6cm]
\tikzstyle{removedvertex} = [white,draw,circle,fill=white,minimum width=0.6cm]
\tikzstyle{ass1} = [densely dotted,very thick]
\tikzstyle{ass2} = [densely dashed,thick]
\tikzstyle{bottleneck} = [red]
\tikzstyle{critical} = [orange]

\begin{tikzpicture}[scale=2.3]
	{\scriptsize
	\node[agent] (a1) {$\alpha_1$};
	\node[removedvertex] (a2) at ($(a1) + (0,-0.5)$) {$\alpha_2$};
	\node[agent] (a3) at ($(a2) + (0,-0.5)$) {$\alpha_3$};
	\node[agent] (a4) at ($(a3) + (0,-0.5)$) {$\alpha_4$};
	
	\node[task] (t1) at ($(a1) + (0.75,-0.25)$){$\tau_1$};
	\node[removedvertex] (t2) at ($(a2) + (0.75,-0.25)$) {$\tau_2$};
	\node[task] (t3) at ($(a3) + (0.75,-0.25)$) {$\tau_3$};

	\draw[ass1,critical] (a1) -- node [above] {4} (t1);
	\draw[ass2] (a1) -- node [below,very near start] {2} (t3);
	
	
	\draw[] (a3) -- node [above,near end] {7} (t1);
	\draw[] (a3) -- node [below,near end] {8} (t3);
	
	\draw[ass2,bottleneck,very thick] (a4) -- node [above,very near start] {2} (t1);
	\draw[ass1] (a4) -- node [below] {1} (t3);
	}
\end{tikzpicture}
        \caption{Order $k=2$} 
        \label{fig:LexExample_k2}
    \end{subfigure}
    \hfill
    \begin{subfigure}{0.32\linewidth}
        \centering
         \tikzstyle{agent} = [draw,circle,fill=blue!20,minimum width=0.6cm]
\tikzstyle{task} = [draw,circle,fill=green!20,minimum width=0.6cm]
\tikzstyle{removedvertex} = [white,draw,circle,fill=white,minimum width=0.6cm]
\tikzstyle{ass1} = [densely dotted,very thick]
\tikzstyle{ass2} = [densely dashed,thick]
\tikzstyle{bottleneck} = [red]
\tikzstyle{critical} = [orange]

\begin{tikzpicture}[scale=2.3]
	{\scriptsize
	\node[agent] (a1) {$\alpha_1$};
	\node[removedvertex] (a2) at ($(a1) + (0,-0.5)$) {$\alpha_2$};
	\node[agent] (a3) at ($(a2) + (0,-0.5)$) {$\alpha_3$};
	\node[removedvertex] (a4) at ($(a3) + (0,-0.5)$) {$\alpha_4$};
	
	\node[removedvertex] (t1) at ($(a1) + (0.75,-0.25)$){$\tau_1$};
	\node[removedvertex] (t2) at ($(a2) + (0.75,-0.25)$) {$\tau_2$};
	\node[task] (t3) at ($(a3) + (0.75,-0.25)$) {$\tau_3$};

	\draw[ass2,bottleneck,very thick] (a1) -- node [below,very near start] {2} (t3);
	
	
	\draw[critical] (a3) -- node [below,near end] {8} (t3);
	
	}
\end{tikzpicture}
        \caption{Order $k=3$} 
        \label{fig:LexExample_k3}
    \end{subfigure}
    \caption{Subgraphs, $\subgraph_k=(\subagents_k,\subtasks_k,\subedges_k)$, with edge weights, $\weights$, for sequential bottleneck assignment of tasks, $\tasks=\{\tau_l\}_{l=1}^3$ [green nodes], to agents, $\agents=\{\alpha_l\}_{l=1}^4$ [blue nodes]. 
    For each order, $k\in\{1,2,3\}$, a maximum-margin bottleneck edge, $(i^*_k,j^*_k)\in\ebottmm_{\subagents_k,\subtasks_k}(\weights)$ [red lines], and a critical edge associated with the robustness margin, $(i^c_k,j^c_k)\in\allebott_{\subagents_k,\subtasks_k}(\subedges_k\setminus\{(i^*_k,j^*_k)\},\weights)$ [orange lines], are shown. A bottleneck minimising assignment, $\hat{\Pi}\in\BAop_{\agents,\tasks}(\edges,\weights)$ [dotted], and the sequential bottleneck optimising assignment, $\Pi'\in\LexAop_{\agents,\tasks}(\weights)\subset\BAop_{\agents,\tasks}(\edges,\weights)$ [dashed], are highlighted.}
    \label{fig:LexExample}
\end{figure}

We consider the sequential bottleneck assignment of $n=3$ tasks, $\tasks=\{\tau_l\}_{l=1}^3$, to $m=4$ agents, $\agents=\{\alpha_l\}_{l=1}^4$, with assignment weights, $\weights$, illustrated in~\cref{fig:LexExample}. 
The sequence of \ac{bap}s is initialised with the full assignment graph, $\subgraph_1 = \graph = (\agents,\tasks,\edges)$, with $\edges=\agents\times\tasks$, as shown in~\cref{fig:LexExample_k1}. The bottleneck weight is $\Bop_{\agents,\tasks}(\edges,\weights)=4$. There are two edges that have weight equal to the bottleneck, $\allebott_{\agents,\tasks}(\edges,\weights)=\{(\alpha_1,\tau_1),(\alpha_2,\tau_2)\}$, and one of them is the unique maximum-margin bottleneck edge, $(i^*_1,j^*_1)=(\alpha_2,\tau_2)=\ebottmm_{\agents,\tasks}(\weights)$. That is, $\Bop_{\agents,\tasks}(\edges\setminus\{(\alpha_1,\tau_1)\},\weights)=4$, 
$\Bop_{\agents,\tasks}(\edges\setminus\{(\alpha_2,\tau_2)\},\weights)=7$, and the 
first order robustness margin is therefore $\mu_1 = \robOp_{\agents,\tasks}(\weights)=3$. Because the 
robustness margin is greater than zero, all bottleneck minimising assignments allocate $\alpha_2$ to $\tau_2$. We note that there exist multiple valid assignments with maximum weight equal to the bottleneck. 
Assignment $\hat{\Pi}$, corresponding to the task-agent pairings $\{(\alpha_1,\tau_1),(\alpha_2,\tau_2),(\alpha_4,\tau_3)\}$, and assignment $\Pi'$, corresponding to 
$\{(\alpha_1,\tau_3),(\alpha_2,\tau_2),(\alpha_4,\tau_1)\}$, both result in a bottleneck weight of 4, i.e., $\hat{\Pi},\Pi'\in\BAop_{\agents,\tasks}(\edges,\weights)$. These bottleneck minimising assignments differ in the higher order criteria of the sequential bottleneck assignment however.

By removing the first order bottleneck agent and task, 
we obtain the second order subgraph, $\subgraph_2=(\subagents_2,\subtasks_2,\subedges_2)$, with $\subagents_2=\agents\setminus\{\alpha_2\}$, $\subtasks_2=\tasks\setminus\{\tau_2\}$, $\subedges_2=\subagents_2\times\subtasks_2$, shown in~\cref{fig:LexExample_k2}. The second order bottleneck weight is $\Bop_{\subagents_2,\subtasks_2}(\subedges_2,\weights)=2$. There are two edges with weight equal to the bottleneck, $\allebott_{\subagents_2,\subtasks_2}(\subedges_2,\weights)=\{(\alpha_1,\tau_3),(\alpha_4,\tau_1)\}$, and both are maximum-margin bottlenecks, $\ebottmm_{\subagents_2,\subtasks_2}(\weights)=\allebott_{\subagents_2,\subtasks_2}(\subedges_2,\weights)$, with $\Bop_{\subagents_2,\subtasks_2}(\subedges_2\setminus\{(\alpha_1,\tau_3)\},\weights)=\Bop_{\subagents_2,\subtasks_2}(\subedges_2\setminus\{(\alpha_4,\tau_1)\},\weights)=4$. From Proposition~\ref{prop:UniqueMMB} it follows that all sequential bottleneck optimal assignments involve the agent-task pairings corresponding to both edges in $\ebottmm_{\subagents_2,\subtasks_2}(\weights)$. From these two edges, we arbitrarily select the second order bottleneck edge, $(i^*_2,j^*_2)=(\alpha_4,\tau_1)$, with 
robustness margin, $\mu_2=\robOp_{\subagents_2,\subtasks_2}(\weights)=2$. In the final step we consider the subgraph, $\subgraph_3=(\subagents_3,\subtasks_3,\subedges_3)$, with $\subagents_3=\subagents_2\setminus\{\alpha_4\}$, $\subtasks_2=\subtasks_2\setminus\{\tau_1\}$, $\subedges_3=\subagents_3\times\subtasks_3$, shown in~\cref{fig:LexExample_k3}. The third order bottleneck weight is, $\Bop_{\subagents_3,\subtasks_3}(\subedges_3,\weights)=2$, with unique bottleneck edge, $(i^*_3,j^*_3)=\allebott_{\subagents_3,\subtasks_3}(\subedges_3,\weights)=\ebottmm_{\subagents_3,\subtasks_3}(\weights)=(\alpha_1,\tau_3)$, and robustness margin, $\mu_3=\robOp_{\subagents_3,\subtasks_3}(\weights)=6$. 

The resulting sequential bottleneck optimising assignment, $\Pi^*=\Pi'\in\LexAop_{\subagents_3,\subtasks_3}(\weights)$, is unique and a robust lexicographical assignment because the robustness margins are strictly positive for all orders, i.e., every other admissible assignment, $\Pi\in\vass_{\agents,\tasks}(\edges)\setminus\{\Pi^*\}$, has a lexicographically larger weight sequence than $(w_{i^*_1,j^*_1},w_{i^*_2,j^*_2},w_{i^*_3,j^*_3})=(4,2,2)$. Furthermore, we observe that $w_{i^*_1,j^*_1}+\mu_1<\max\{w_{i^*_1,j^*_2},w_{i^*_2,j^*_1}\} = 8$, $w_{i^*_1,j^*_1}+\mu_1<\max\{w_{i^*_1,j^*_3},w_{i^*_3,j^*_1}\} = 9$, and $w_{i^*_2,j^*_2}+\mu_2=\max\{w_{i^*_2,j^*_3},w_{i^*_3,j^*_2}\} = 4 $ in accordance with Proposition~\ref{prop:AssEdgeSwitch}. We also see that $w_{i^*_1,j^*_1}+\mu_1<w_{\alpha_3,j^*_1}=9$, $w_{i^*_2,j^*_2}+\mu_2<w_{\alpha_3,j^*_2}=7$, and
$w_{i^*_3,j^*_3}+\mu_3=w_{\alpha_3,j^*_3}=8$ in agreement with Proposition~\ref{prop:UnassAgentSwitch}.


\section{Collision Avoidance}\label{sec:CollisionAvoidance} In this section we address inter-agent collision avoidance. We first investigate sufficient conditions 
related to the sequential bottleneck assignment 
 in Section~\ref{sec:GeneralConditions}. Then, in Section~\ref{sec:LocalConstratints} we introduce time-dependent position constraints for the individual agents.
 
\subsection{Sufficient Conditions for Collision Avoidance}\label{sec:GeneralConditions}
\begin{assum} \label{ass:asslex} 
The assignment weights, $\weights=\{w_{i,j}=\dist(p_i(0),g_j)\,|\,(i,j)\in\agents\times\tasks\}$, are the distances between initial agent positions and destinations, defined in~\cref{ass:dist}. Agents, $\agents$, are allocated to destinations, $\tasks$, with a sequential bottleneck optimising 
assignment, $\Pi^*\in\LexAop_{\agents,\tasks}(\weights)$, as in Definition~\ref{def:lexass}, with $k$-th order bottleneck agent-task pair, $(i^*_k,j^*_k)$, and 
robustness margin, $\mu_k$, for 
$k\in\{1,\dots,n\}$.
\end{assum}

 Using the safety distances from Definition~\ref{defn:safetydistance}, we provide a first condition which guarantees that an assigned agent does not collide with any other agent at a particular time.
\begin{lem} \label{lem:generalsourcecond}
  Given \cref{ass:dist,,ass:asslex}. The $k$-th order bottleneck agent, $i^*_k$, $k\in\{1,\dots,n\}$, 
  does not collide with any other agent, $i'\in\agents\setminus\{i^*_k\}$, at time $t$ if  
  \begin{align}\label{eq:targetradcond}
    \dist(p_{i'}(0),p_{i'}(t)) +\dist(p_{i^*_k}(t),g_{j^*_k}) &< w_{i',j^*_k} - s_{i',i^*_k}\,.
  \end{align}
\end{lem}
\begin{proof}
  Considering the triangle inequality in~\eqref{eq:triangleIneq}, the distance between the initial position of agent $i'$ and the target destination assigned to agent $i^*_k$ is bounded, $\dist(p_{i'}(0),g_{j^*_k}) \leq \dist(p_{i'}(0),p_{i'}(t)) + \dist(p_{i'}(t),p_{i^*_k}(t)) + \dist(p_{i^*_k}(t),g_{j^*_k})$. By applying \eqref{eq:targetradcond}, with $w_{i',j^*_k} = \dist(p_{i'}(0),g_{j^*_k})$, we obtain $\dist(p_{i'}(t),p_{i^*_k}(t)) > s_{i',i^*_k}$ as required from Definition~\ref{defn:safetydistance}. 
  %
\end{proof}


For guaranteed collision avoidance among assigned agents we combine the concept of robustness margin with the condition given in Lemma~\ref{lem:generalsourcecond}.
\begin{prop}\label{prop:CollisionAvoidanceGeneral}
   Given \cref{ass:dist,,ass:asslex}. The $k$-th order bottleneck agent, $i^*_k$, $k\in\{1,\dots,n-1\}$, 
   does not collide with a higher order bottleneck agent, $i^*_l$, $l\in\{k + 1,\dots n\}$, at time $t$ if both following conditions are satisfied,
   \begin{subequations}\label{eq:conditions}
   \begin{align}
      \dist(p_{i^*_l}(0),p_{i^*_l}(t)) + \dist(p_{i^*_k}(t),g_{j^*_k}) &< w_{i^*_k,j^*_k} + \mu_k - s_{i^*_l,i^*_k}\,,\label{eq:cond1} \\  
      \dist(p_{i^*_k}(0),p_{i^*_k}(t)) + \dist(p_{i^*_l}(t),g_{j^*_l}) &< w_{i^*_k,j^*_k}  + \mu_k - s_{i^*_l,i^*_k} \,.\label{eq:cond2}
  \end{align}
   \end{subequations}
\end{prop}
\begin{proof}
  For an arbitrary order, $k\in\{1,\dots,n-1\}$, let $l\in\{k + 1,\dots n\}$ be an arbitrary higher order. If $w_{i^*_l,j^*_k} > \dist(p_{i^*_l}(0),p_{i^*_l}(t)) + \dist(p_{i^*_k}(t),g_{j^*_k}) +  s_{i^*_l,i^*_k}$, agents $i^*_k$ and $i^*_l$ do not collide at time~$t$ as shown in Lemma~\ref{lem:generalsourcecond}, with agent~$i^*_l$ playing the role of the 'other agent'.

It remains to consider the case where 
$w_{i^*_l,j^*_k} \leq \dist(p_{i^*_l}(0),p_{i^*_l}(t)) + \dist(p_{i^*_k}(t),g_{j^*_k}) +  s_{i^*_l,i^*_k}$. Assume for the sake of contradiction that the condition in~\eqref{eq:targetradcond} also does not hold from perspective of agent~$i^*_l$, where agent~$i^*_k$ is the 'other agent', i.e., 
$w_{i^*_k,j^*_l} \leq \dist(p_{i^*_k}(0),p_{i^*_k}(t)) + \dist(p_{i^*_l}(t),g_{j^*_l}) + s_{i^*_k,i^*_l}$. Then, from~\eqref{eq:cond1} we have $w_{i^*_l,j^*_k} < w_{i^*_k,j^*_k}+\mu_k$ and from~\eqref{eq:cond2} we have $w_{i^*_k,j^*_l} < w_{i^*_k,j^*_k}+\mu_k$. 
This however contradicts Proposition~\ref{prop:AssEdgeSwitch}. 
It follows that $w_{i^*_k,j^*_l} > \dist(p_{i^*_k}(0),p_{i^*_k}(t)) + \dist(p_{i^*_l}(t),g_{j^*_l}) + s_{i^*_k,i^*_l}$. Thus, according to Lemma~\ref{lem:generalsourcecond}, agents $i^*_k$ and $i^*_l$ do not collide at time~$t$. 
\end{proof}

\subsection{Local Constraints for Guaranteed Collision Avoidance}\label{sec:LocalConstratints}
We now derive individual position constraints for every agent such that if all agents satisfy their associated constraints, collisions involving assigned agents are avoided. The constraints rely on the robustness of the sequential bottleneck assignment as formalised in the following assumption. 
\begin{assum}\label{ass:globalsafedist}
  There exists an upper bound on the safety distance, $s\geq s_{i,i'}$, between all agents, $i,i'\in\agents$, $i\neq i'$, that is smaller than the $k$-th order bottleneck robustness margin for all $k\in\{1,\dots,n\}$, i.e.,
  \begin{align}\label{eq:minmargin}
    s< \mu:=\min_{k\in\{1,\dots,n\}} \mu_k\,.
  \end{align}
\end{assum}

The agent position constraints are composed of up to two components. The first component is a bound on the distance of an agent from its initial position. This bound is limited by the same time-varying parameter, $a(t)$, for all agents up to a  saturation value, $A_k$, that is agent dependent. 
\begin{assum}\label{ass:boundfromsource}
  The distance of an agent $i$, to its initial position is bounded by $\ds_i(t)$, i.e., 
  \begin{align*}
      \dist(p_{i}(0),p_{i}(t))< \ds_i(t)\,,
  \end{align*}
 for all $i\in\agents$ and $t\in[0,T]$, with
\begin{align*}
  \ds_{i^*_k}(t) =
  \begin{cases}
    a(t) &\textrm{if }a(t) \leq A_k\,,\\
    A_k & \textrm{otherwise}\,,
  \end{cases}
\end{align*} for all assigned agents $i^*_k$, $k\in\{1,\dots,n\}$, and $\ds_{i'}(t) = 
  a_{i^*_n}(t)$, for all unassigned agents, $i'\in\agents\setminus\{i^*_1,\dots,i^*_n\}$, where
\begin{align}\label{eq:saturationvalue}
  A_k := \min_{l\in\{1,\dots,k\}}w_{i^*_l,j^*_l} + \mu_l - \frac{1}{2}(\mu + s)
\end{align}
 and $a(t)\geq \frac{1}{2}(\mu-s)$.
\end{assum}

The second component of the position constraints applies only to assigned agents. It consists of a bound on the distance of an agent to its target destination that decreases when the bound on the distance from the initial position increases. 
\begin{assum}\label{ass:boundfromtraget}
The distance of an assigned agent, $i^*_k$, to its destination is bounded by $\dt_{j^*_k}(t)$, i.e., 
\begin{align*}
    \dist(p_{i^*_k}(t),g_{j^*_k})< \dt_{j^*_k}(t)\,,
\end{align*} for all assigned agents $i^*_k$, $k\in\{1,\dots,n\}$, and $t\in[0,T]$, with
\begin{align*}
  \dt_{j^*_k}(t) = A_k - \ds_{i^*_k}(t) + \frac{1}{2}(\mu - s)\,.
\end{align*}
\end{assum}


Satisfaction of these local bounds provides a sufficient condition for collision avoidance as shown in the following.
\begin{thm}\label{thm:CollisionAvoidance}
  Given $n$ target destinations allocated to $m$ agents with a sequential bottleneck optimising assignment, as specified in Assumption~\ref{ass:asslex}, with robustness margins satisfying Assumption~\ref{ass:globalsafedist}, based on distance weights introduced in Assumption~\ref{ass:dist}. No assigned agent, $i^*_k$, $k\in\{1,\dots,n\}$, 
  collides with any other agent, $i'\in\agents$, at any time, $t\in[0,T]$, 
  if all agents satisfy the position bounds of Assumption~\ref{ass:boundfromsource} and all assigned agents additionally satisfy the position bounds of Assumption~\ref{ass:boundfromtraget}. 
\end{thm}
\begin{proof}
  By construction, the bounds in \cref{ass:boundfromsource,,ass:boundfromtraget} satisfy $a_{i^*_k}(t)\geq a_{i^*_l}(t)$, $b_{j^*_k}(t)\geq b_{j^*_l}(t)$, and $a_{i^*_k}(t) + b_{j^*_k}(t) \leq w_{i^*_k,j^*_k} + \mu_k - s$ for all $k\in\{1,\dots,n\}$, $l\in\{k,\dots,n\}$, and $t\in[0,T]$. With Assumption~\ref{ass:globalsafedist} 
  the conditions in Proposition~\ref{prop:CollisionAvoidanceGeneral} 
  are satisfied for all $k\in\{1,\dots,n-1\}$ and $t\in[0,T]$. Thus, no assigned agents collide with each other. Because of Proposition~\ref{prop:UnassAgentSwitch},  Lemma~\ref{lem:generalsourcecond} is satisfied for all $k\in\{1,\dots,n\}$, $i'\in\agents\setminus\{i^*_1,\dots,i^*_n\}$, and $t\in[0,T]$. Thus, assigned agents also do not collide with unassigned agents.
\end{proof}

We note that there exist positions, $p_{i^*_k}(t)$, for every assigned agent, $i^*_k$, $k\in\{1,\dots,n\}$, at every time, $t\in[0,T]$, that satisfy \cref{ass:boundfromsource,,ass:boundfromtraget} if Assumption~\ref{ass:globalsafedist} holds. That is, Assumption~\ref{ass:boundfromsource} bounds the position, $p_{i^*_k}(t)$, to lie within a ball centred at the initial position, $p_{i^*_k}(0)$, with radius $a_{i^*_k}(t)\in[\frac{1}{2}(\mu-s),A_k]$  and Assumption~\ref{ass:boundfromtraget} bounds the position, $p_{i^*_k}(t)$, to lie within a ball centred at the target destination, $g_{j^*_k}$, with radius $b_{j^*_k}(t)\in[\frac{1}{2}(\mu-s),A_k]$. The two balls intersect for all $t\in[0,T]$ if 
$\mu>s$ because the sum of the radii is larger than the distance between the centres, 
\begin{align*}
    a_{i^*_k}(t) + b_{j^*_k}(t) 
    \geq w_{i^*_k,j^*_k} + \mu - s 
    > w_{i^*_k,j^*_k}=\dist(p_{i^*_k}(0),g_{j^*_k})\,.
\end{align*}   
It follows that the constrained sets of safe positions constructed from the bounds in \cref{ass:boundfromsource,,ass:boundfromtraget}, i.e.,
\begin{subequations}\label{eq:constraints}
  \begin{align}
    \locCon_{i^*_k}(t)=\big\{p\in\R^{n_p}\,\big|\,&\dist(p_{i^*_k}(0),p)<a_{i^*_k}(t)\,,\\
    &\dist(p,g_{j^*_k})<b_{j^*_k}(t)\big\}\,,\nonumber
\end{align} for all assigned agents $i^*_k$, $k\in\{1,\dots,n\}$, 
and 
\begin{align}
    \locCon_{i'}(t)=\big\{p\in\R^{n_p}\,\big|\,&\dist(p_{i'}(0),p)<a_{i'}(t)\big\}\,,
\end{align}
\end{subequations} 
for unassigned agents, $i'\in\agents\setminus\{i^*_1,\dots,i^*_n\}$, are all non-empty if Assumption~\ref{ass:globalsafedist} holds. 

For any agent, $i\in\agents$, the safe set, $\locCon_i(t)$, depends on the timing parameter, $a(t)$, the $k$-th order bottleneck weights and robustness margins obtained in the sequential bottleneck assignment, 
$w_{i^*_k,j^*_k}$ and $\mu_k$ respectively, for $k\in\{1,\dots,n\}$, the initial position, $p_{i}(0)$, and 
the location of its assigned destination, $g_{j}$, if $\pi^*_{i,j}=1$. We see that the larger the robustness margins are, the less the agent positions need to be constrained in~\eqref{eq:constraints}. 
These local position constraints 
provide a solution for the problem outlined in Section~\ref{sec:ProblemFormulation}.
\begin{cor}
The agent position constraints given in~\eqref{eq:constraints} can be obtained with a complexity $\bigO(mn^{3})$ via a sequential bottleneck assignment. 
The resulting local safe sets are a solution for Problem~\ref{prob:main} if Assumption~\ref{ass:globalsafedist} holds.
\end{cor}

 An optimising assignment, $\Pi^*\in\LexAop_{\agents,\tasks}(\weights)$, and the 
 robustness margins are determined only based on information of the initial agent positions relative to the target destinations. Knowledge of the absolute positions of the other agents is not required to compute the individual constraints.
\begin{rem}\label{rem:DistConstraints}
The collision avoidance constraints given in~\eqref{eq:constraints} can be obtained with distributed computation. Agents 
need to coordinate through the shared scheduling variable, $a(t)$, and by exchanging estimates of all the $k$-order bottleneck distances and robustness margins, for all $k\in\{1,\dots,n\}$, during the assignment as suggested in Remark~\ref{rem:DistLex}.
\end{rem}

The local agent position constraints proposed in~\eqref{eq:constraints} can be incorporated in many different motion control or trajectory planning applications. For instance, because the position bounds provide collision avoidance guarantees without fully specifying position trajectories, they can be included in predictive optimisation approaches with objective functions that do not consider the coordination among the agents. The resulting optimisation problems can incorporate additional constraints such as avoidance of 
other objects. The sets in~\eqref{eq:constraints}
are convex as they are constructed from distance functions. This allows to bypass an extra procedure for convex approximation of a safe region which is typically required in model predictive motion control, see \cite{Bemporad2011CDC} for instance. For specific choices of the applied distance functions, e.g., the 1-norm or the infinity-norm distances, the constraints are linear in the position variables and can be efficiently encoded. 
The time-varying position bounds can also be used to verify that motion control strategies derived from simplified assumptions do not result in collisions in practice or in higher fidelity simulations. 
We investigate an example in the following Section. 

\section{Case Study}\label{sec:CaseStudy}

\begin{figure*}[t]
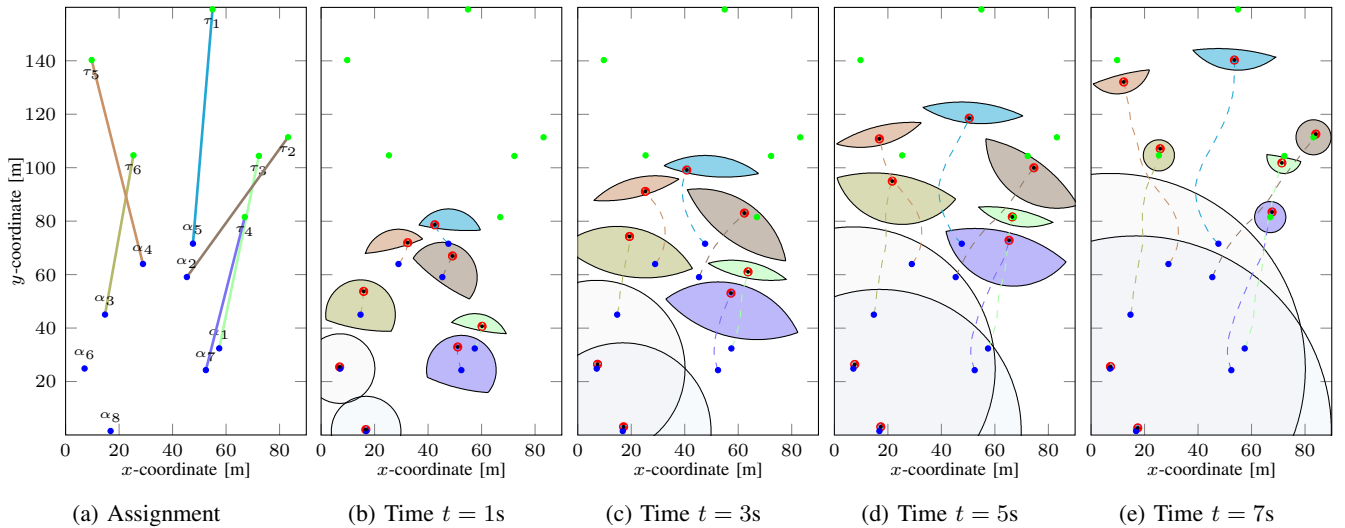

 \setlength{\axiswidth}{0.18\textwidth}
 \setlength{\axisheight}{0.32\textwidth}
 \centering
\begin{subfigure}{0.22\textwidth}
  \centering
%
%
\definecolor{mycolor1}{rgb}{0.66075,0.99063,0.65692}%
\definecolor{mycolor2}{rgb}{0.57260,0.48991,0.41615}%
\definecolor{mycolor3}{rgb}{0.70594,0.72129,0.41299}%
\definecolor{mycolor4}{rgb}{0.79032,0.57376,0.40566}%
\definecolor{mycolor5}{rgb}{0.12355,0.65430,0.83664}%
\definecolor{mycolor6}{rgb}{0.47967,0.44597,0.95705}%
\begin{tikzpicture}

\begin{axis}[%
width=\axiswidth,
height=\axisheight,
at={(0\axiswidth,0\axisheight)},
scale only axis,
xmin=0,
xmax=90,
xtick={ 0, 20, 40, 60, 80},
xlabel style={font=\color{white!15!black}},
xlabel={{\scriptsize $x$-coordinate [m]}},
ymin=0,
ymax=160,
ytick={ 20,  40,  60,  80, 100, 120, 140},
ylabel style={font=\color{white!15!black}},
ylabel={{\scriptsize $y$-coordinate [m]}},
axis background/.style={fill=white},
xlabel style={font=\scriptsize,at={(axis description cs:0.5,0.05)}},xticklabel style = {font=\scriptsize},ylabel style={font=\scriptsize,at={(axis description cs:0.18,0.5)}},yticklabel style={font=\scriptsize}
]
\addplot [color=green, draw=none, mark size=1.0pt, mark=*, mark options={solid, fill=green, green}, forget plot]
  table[row sep=crcr]{%
54.9381814754719	159.231300093603\\
83.1786137178556	111.400420258476\\
72.2721481740512	104.417351212404\\
66.9886168039705	81.5459252343556\\
9.74986732028404	140.278595188186\\
25.3581476829567	104.640010843739\\
};
\addplot [color=blue, draw=none, mark size=1.0pt, mark=*, mark options={solid, fill=blue, blue}, forget plot]
  table[row sep=crcr]{%
57.4429325168031	32.3985308545857\\
45.3422732646511	59.0858307240945\\
14.7439326022129	45.0432828104211\\
28.9341032861871	63.9871736496042\\
47.6164399484537	71.5878522122779\\
7.09136385210165	24.866030105122\\
52.4503263961034	24.2790330739011\\
16.8304855305558	1.47968860800485\\
};
\addplot [color=mycolor1, line width=1.0pt, forget plot]
  table[row sep=crcr]{%
72.2721481740512	104.417351212404\\
57.4429325168031	32.3985308545857\\
};
\node[below, align=center]
at (axis cs:72.272,104.417) {{\tiny $\tau_3$}};
\node[above, align=center]
at (axis cs:57.443,32.399) {{\tiny $\alpha_1$}};
\addplot [color=mycolor2, line width=1.0pt, forget plot]
  table[row sep=crcr]{%
83.1786137178556	111.400420258476\\
45.3422732646511	59.0858307240945\\
};
\node[below, align=center]
at (axis cs:83.179,111.4) {{\tiny $\tau_2$}};
\node[above, align=center]
at (axis cs:45.342,59.086) {{\tiny $\alpha_2$}};
\addplot [color=mycolor3, line width=1.0pt, forget plot]
  table[row sep=crcr]{%
25.3581476829567	104.640010843739\\
14.7439326022129	45.0432828104211\\
};
\node[below, align=center]
at (axis cs:25.358,104.64) {{\tiny $\tau_6$}};
\node[above, align=center]
at (axis cs:14.744,45.043) {{\tiny $\alpha_3$}};
\addplot [color=mycolor4, line width=1.0pt, forget plot]
  table[row sep=crcr]{%
9.74986732028404	140.278595188186\\
28.9341032861871	63.9871736496042\\
};
\node[below, align=center]
at (axis cs:9.75,140.279) {{\tiny $\tau_5$}};
\node[above, align=center]
at (axis cs:28.934,63.987) {{\tiny $\alpha_4$}};
\addplot [color=mycolor5, line width=1.0pt, forget plot]
  table[row sep=crcr]{%
54.9381814754719	159.231300093603\\
47.6164399484537	71.5878522122779\\
};
\node[below, align=center]
at (axis cs:54.938,159.231) {{\tiny $\tau_1$}};
\node[above, align=center]
at (axis cs:47.616,71.588) {{\tiny $\alpha_5$}};
\node[above, align=center]
at (axis cs:7.091,24.866) {{\tiny $\alpha_6$}};
\addplot [color=mycolor6, line width=1.0pt, forget plot]
  table[row sep=crcr]{%
66.9886168039705	81.5459252343556\\
52.4503263961034	24.2790330739011\\
};
\node[below, align=center]
at (axis cs:66.989,81.546) {{\tiny $\tau_4$}};
\node[above, align=center]
at (axis cs:52.45,24.279) {{\tiny $\alpha_7$}};
\node[above, align=center]
at (axis cs:16.83,1.48) {{\tiny $\alpha_8$}};
\end{axis}
\end{tikzpicture}%
  \caption{Assignment}
  \label{fig:setup}
\end{subfigure}
\begin{subfigure}{0.185\textwidth}
  \centering
  \input{images/trajectories_t1.tikz}
  \caption{Time $t=1$s}
  \label{fig:traj_t1}
\end{subfigure}
\begin{subfigure}{0.185\textwidth}
  \centering
  \input{images/trajectories_t3.tikz}
  \caption{Time $t=3$s}
  \label{fig:traj_t3}
\end{subfigure}
\begin{subfigure}{0.185\textwidth}
  \centering
  \input{images/trajectories_t5.tikz}
  \caption{Time $t=5$s}
  \label{fig:traj_t5}
\end{subfigure}
\begin{subfigure}{0.185\textwidth}
  \centering
  \input{images/trajectories_t7.tikz}
  \caption{Time $t=7$s}
  \label{fig:traj_t7}
\end{subfigure}
  \caption[]{Case study with agents $\agents=\{\alpha_l\}^8_{l=1}$ and tasks $\tasks=\{\tau_l\}^6_{l=1}$. 
  The sequential bottleneck optimising assignment is illustrated in (a). The positions satisfying the 
  constraints, $\locCon_i(t)$, are shown in~(b)-(e) as shaded areas for every agent, $i\in\agents$, 
  with initial positions, $p_i(0)$~%
  [blue dots]%
  ,  positions at the time~$t$, $p_i(t)$~%
  [
  black dots]%
  , surrounded by safety circle with diameter $s$~%
  [red circles]%
  , past trajectories, $p_i(t')$, $t'\in[0,t]$~%
  [dashed lines]%
  , and with 
  destinations, $g_j$~%
  [green dots]%
  , for all tasks $j\in\tasks$.}
  \label{fig:bounds}
\end{figure*}

We consider $m=8$ mobile robots, $\agents=\{\alpha_l\}^8_{l=1}$, operating on a plane. The location of each agent, $i\in\agents$, at time~$t$ is represented by a position point, $p_i(t)=(x_i(t),y_i(t))\in\R^{n_p}$ with $n_p=2$. We assume there are  $n=6$ tasks, $\tasks=\{\tau_l\}^6_{l=1}$, corresponding to the actions of visiting target destinations, $\{g_j\in\R^{n_p} \,|\,j\in\tasks\}$, that are each to be completed by one agent given the initial positions $\{p_i(0)\in\R^{n_p}\,|\,i\in\agents\}$.  \cref{fig:setup} illustrates the configuration at time $t=0$. The tasks are allocated to the agents according to a sequential bottleneck assignment, as proposed in Assumption~\ref{ass:asslex}, with the distance function in Assumption~\ref{ass:dist} specified to be the Euclidean distance,  $\dist(p,p') =\|p'-p\|_2$. 
The resulting optimising assignment, $\Pi^*\in\LexAop_{\agents,\tasks}(\weights)$, is described in Table~\ref{tab:lexassignment}, with $k$-order bottleneck edges, defined in~\eqref{eq:kthbott}, 
and robustness margins, defined in~\eqref{eq:kthmargin}, for all $k\in\{1,\dots,6\}$. We note that for every order the set of maximum-margin bottleneck edges is a singleton, all robustness margins are strictly positive, and that $\Pi^*$ is therefore a robust lexicographic assignment.

\definecolor{mycolor1}{rgb}{0.66075,0.99063,0.65692}%
\definecolor{mycolor2}{rgb}{0.57260,0.48991,0.41615}%
\definecolor{mycolor3}{rgb}{0.70594,0.72129,0.41299}%
\definecolor{mycolor4}{rgb}{0.79032,0.57376,0.40566}%
\definecolor{mycolor5}{rgb}{0.12355,0.65430,0.83664}%
\definecolor{mycolor6}{rgb}{0.69953,0.64656,0.68336}%
\definecolor{mycolor7}{rgb}{0.47967,0.44597,0.95705}%
\definecolor{mycolor8}{rgb}{0.41407,0.58609,0.87963}%

\begin{table}[ht]
  \caption{Sequential bottleneck optimising assignment of target destinations to agents with initial positions and colour coding shown in~\cref{fig:setup}.}\label{tab:lexassignment}
  \begin{center}
   \begin{tabular}{>{\raggedleft\arraybackslash}p{0.6cm}|>{\raggedleft\arraybackslash}p{1.2cm}|>{\raggedleft\arraybackslash}p{1.2cm}|>{\raggedleft\arraybackslash}p{1.2cm}|>{\raggedleft\arraybackslash}p{0.8cm}}
    Order & {Bottleneck edge} & Bottleneck weight & Robustness margin & 
    Bound limit\\
    $k$ & $(i^*_k,j^*_k)$ & $w_{i^*_k,j^*_k}$ & $\mu_k$ & 
    $A_k$\\
    \hline
    \rowcolor{mycolor5!50}
    1 & $(\alpha_5,\tau_1)$ & 87.95m & 10.78m & 
    92.72m\\
    \rowcolor{mycolor4!50}
    2 & $(\alpha_4,\tau_5)$ & 78.67m & 9.99m & 
    82.64m \\
    \rowcolor{mycolor1!50}
    3 & $(\alpha_1,\tau_3)$ & 73.53m & 9.02m & 
    76.54m \\
    \rowcolor{mycolor2!50}
    4 & $(\alpha_2,\tau_2)$ & 64.56m & 27.82m & 
    76.54m \\
     \rowcolor{mycolor3!50}
    5 & $(\alpha_3,\tau_6)$ & 60.53m & 21.30m & 
    75.83m \\
    \rowcolor{mycolor7!50}
    6 & $(\alpha_7,\tau_4)$ & 59.08m & 23.38m & 
    75.83m \\
  \end{tabular}
  \end{center}
\end{table}


The agents are of varying size but we assume that no agent occupies any space outside of a radius of 1.5m from its position point, i.e., at time $t$ the body of agent, $i\in\agents$, lies within a safety circle with diameter $s=3$m centred at $p_i(t)$. The smallest 
robustness margin of the 
assignment, defined in~\eqref{eq:minmargin}, is 
$\mu=9.02$m. 
Because $s<\mu$, Assumption~\ref{ass:globalsafedist} is satisfied and the safe sets, $\locCon_i(t)\subset\R^2$, in~\eqref{eq:constraints} are non-empty for all agents $i\in\agents$.
\cref{fig:traj_t1,fig:traj_t3,fig:traj_t5,fig:traj_t7} shows the time-varying areas that satisfy the position constraints for every agent 
given the time-dependent coordination parameter, 
\begin{align*}
    a(t)=\vref t + \frac{1}{2}(\mu - s)\,,
\end{align*} evolving with constant rate, $\vref=10$m/s. We observe that for any assigned agent, $i^*_k$, $k\in\{1,\dots,6\}$, there is always at least a distance of $3$m between any point in the safe set, $\locCon_{i^*_k}(t)$, and any point in the safe of any other agent, $\locCon_{i'}(t)$, $i'\in\agents\setminus\{i^*_k\}$. The only safe sets that 
intersect are those of the unassigned agents.
For the assigned agents, $\{\alpha_1,\alpha_2,\alpha_3,\alpha_4\,\alpha_5,\alpha_7\}$, the local constraints consist of the intersection of both bounds introduced in \cref{ass:boundfromsource,,ass:boundfromtraget}. The two bounds overlap by a different amount for every 
agent, $i^*_k$, $k\in\{1,\dots,6\}$, depending on the difference of the corresponding saturation value, $A_k$ defined in~\eqref{eq:saturationvalue}, and the associated weight, $w_{i^*_k,j^*_k}$. We observe from Table~\ref{tab:lexassignment} that the saturation values of agents $\alpha_2$ and $\alpha_7$, corresponding to bottleneck orders $k=4$ and $k=6$ respectively, are not given by their own associated robustness margins but rather by the value of a lower order robustness margin. That is, the sequence of saturation values, $(A_1,\dots,A_6)$, 
is set to be non-increasing in agreement with~\eqref{eq:saturationvalue}. The constraints for the unassigned agents, $\{\alpha_6,\alpha_8\}$, only consist of limits on the distances 
from their initial positions that are bound by $A_6=75.83$m. If the positions of all agents remain inside the individual safe sets at all considered times, it is guaranteed from Theorem~\ref{thm:CollisionAvoidance} that none of the six assigned agents 
collide with any agent, including the two unassigned agents.



 We consider a scenario where the motion of the robotic agents is governed by nonlinear unicycle models. The agents are guided from their initial positions to their assigned destinations with decentralised feedback controllers. The controllers are obtained from an optimal control approach in which the models are linearised around reference straight-line trajectories, with constant reference speed, $v(t)=\vref$, and where input constraints are neglected. \cref{fig:traj_t1,fig:traj_t3,fig:traj_t5,fig:traj_t7} show a sample simulation in which the controllers are applied to the nonlinear models with input constraints, steering rate disturbances, and where the initial heading angles are not aligned with the targets. The resulting agent position trajectories clearly deviate from the straight-line references. However, because the positions of all agents never leave the local safe sets, collisions involving the assigned agents are guaranteed not to occur. It is enough to tune or design the controllers for every agent individually such that the deviations from the straight-line paths do not violate the local constraints. 
 This demonstrates how the motion control of the agents can be decoupled while still guaranteeing collision avoidance. 
 We note that the same analysis and planning can be applied to robots moving in three dimensions, $n_p=3$, without additional complexity.\footnote{A supplementary video available at http://ieeexplore.ieee.org 
 contains animations of this case study and other scenarios in 2D and 3D.}

\section{Conclusions}\label{sec:conclusion}
We derived local time-varying position constraints for agents that are assigned to move to different target 
destinations. We specified conditions for which satisfaction of the constraints guarantees that collisions are avoided. 
The introduced method allocates targets to agents such that the largest distance between an agent and its destination is minimised. There may be agents that are not assigned to any task but all agents are constrained 
in order to maintain a minimal safety distance from any assigned agent. 
The parameters of the local constraints are obtained from a sequential bottleneck assignment procedure and depend only on the distances between agent initial positions and destinations. 
The local constraints define non-empty convex sets of safe positions for every agent if the optimising assignment is sufficiently robust to changes in the distances relative to the size of the agents. We defined a method to quantify the robustness 
and observed that the more robust the assignment is, the larger the regions of safe positions are.

The constraints can be constructed in polynomial time and with distributed computation where agents only access limited information on other agents. The procedure consists of solving multiple coupled \ac{bap}s and involves structure which may be exploited for faster computation in future work.
The proposed constraints 
provide sufficient but not necessary conditions for collision avoidance. 
In cases where not all conditions are satisfied, other strategies may exist to avoid collisions. 
Incorporating other strategies and investigating alternative local constraints derived from the properties of the assignment motivates further research on this topic. 

\appendices
\section{Proof of Proposition~\ref{prop:UniqueMMB}}\label{app:ProofUniqueMMB}
\begin{proof}
From~\eqref{eq:allbottedges} 
we know that $w_{i^*,j^*}=\Bop_{\subagents,\subtasks}(\subedges,\weights)$ for any $(i^*,j^*)\in\ebottmm_{\subagents,\subtasks}(\weights)$. Given $\robOp_{\subagents,\subtasks}(\weights)>0$, we assume for the sake of contradiction that there exist $\Pi\in\BAop_{\subagents,\subtasks}(\subedges,W)$ and $(i^*,j^*)\in\ebottmm_{\subagents,\subtasks}(\weights)$ with $\pi_{i^*,j^*}=0$. From
~\eqref{eq:bottleneckweight} it follows that $\bott(\Pi,\subedges,\weights)=\Bop_{\subagents,\subtasks}(\subedges,\weights)$. Because $\pi_{i^*,j^*}w_{i^*,j^*}=0$ in~\eqref{eq:maxweight}, we have $\bott(\Pi,\subedges,\weights)=\bott(\Pi,\subedges\setminus\{(i^*,j^*)\},\weights)$. From~\eqref{eq:bottleneckweight} it follows that $\bott(\Pi,\subedges\setminus\{(i^*,j^*)\},\weights)\geq\Bop_{\subagents,\subtasks}(\subedges\setminus\{(i^*,j^*)\},\weights)$ and thus $\Bop_{\subagents,\subtasks}(\subedges\setminus\{(i^*,j^*)\},\weights)\leq\Bop_{\subagents,\subtasks}(\subedges,\weights)$. 
We know that $\Bop_{\subagents,\subtasks}(\subedges\setminus\{(i^*,j^*)\},\weights)\geq\Bop_{\subagents,\subtasks}(\subedges,\weights)$ from \eqref{eq:bottleneckweight} because $\{\Pi\in\vass_{\subagents,\subtasks}(\subedges\setminus\{(i^*,j^*)\})|\pi_{i^*,j^*}=0\}\subset\vass_{\subagents,\subtasks}(\subedges)$. It follows 
that $\Bop_{\subagents,\subtasks}(\subedges\setminus\{(i^*,j^*)\},\weights)=\Bop_{\subagents,\subtasks}(\subedges,\weights)$ which contradicts $\robOp_{\subagents,\subtasks}(\weights)=\Bop_{\subagents,\subtasks}(\subedges\setminus\{(i^*,j^*)\},\weights)-\Bop_{\subagents,\subtasks}(\subedges,\weights)>0$.
\end{proof}

\section{Proof of Proposition~\ref{prop:AssEdgeSwitch}}\label{app:ProofPropAssEdgeSwitch}
\begin{proof}
For an arbitrary order, $a\in\{1,\dots,n-1\}$, and an arbitrary higher order, $b\in\{a+1,\dots,n\}$, we consider the 
assignment, $\hat{\Pi}=\{\hat{\pi}_{i,j}\in\{0,1\}\,|\,(i,j)\in\agents\times\tasks\}$, obtained from $\Pi^*=\{\pi^*_{i,j}\in\{0,1\}\,|\,(i,j)\in\agents\times\tasks\}$ by switching the agent-task 
pairings of $(i^*_a,j^*_a)$ and $(i^*_b,j^*_b)$ to $(i^*_a,j^*_b)$ and $(i^*_b,j^*_a)$, i.e.,
\begin{align*}
    \hat{\pi}_{i,j} =  
    \begin{cases}
        0 & \textrm{if }(i,j)=(i^*_a,j^*_a)\textrm{ or }(i,j)=(i^*_b,j^*_b)\,,\\
        1 & \textrm{if }(i,j)=(i^*_a,j^*_b)\textrm{ or }(i,j)=(i^*_b,j^*_a)\,,\\
        \pi^*_{i,j}& \textrm{otherwise}\,.
    \end{cases}
\end{align*}
We note that $(i^*_a,j^*_b),(i^*_b,j^*_a)\in
\subedges_a$ by definition in~\eqref{eq:subgraph_k}. Because $(i^*_a,j^*_b)$ and $(i^*_b,j^*_a)$ are not assigned according to $\Pi^*$, it follows from Definition~\ref{def:lexass} that $\max\{w_{i^*_a,j^*_b},w_{i^*_b,j^*_a}\}\geq \bott(\Pi^*,\subedges_a,\weights)
$. By construction of $\hat{\Pi}$, we know that $\bott(\hat{\Pi},\hat{\edges},\weights)=\bott(\Pi^*,\hat{\edges},\weights)\leq\bott(\Pi^*,\subedges_a,\weights)$, with $\hat{\edges}:=\subedges_a\setminus\{(i^*_a,j^*_a),(i^*_b,j^*_b),(i^*_a,j^*_b),(i^*_b,j^*_a)\}$. It follows that $\bott(\hat{\Pi},\subedges_a,\weights)
=\max\{w_{i^*_a,j^*_b},w_{i^*_b,j^*_a}\}$. 
Since $(i^*_a,j^*_a)$ is not assigned according to  $\hat{\Pi}$, we know 
$\bott(\hat{\Pi},\subedges_a\setminus\{(i^*_a,j^*_a)\},\weights)=\max\{w_{i^*_a,j^*_b},w_{i^*_b,j^*_a}\}$. Finally, from \cref{def:bottass,,def:criticaledge} we have that $w_{i^*_a,j^*_a}+\mu_a=\Bop_{\subagents_a,\subtasks_a}(\subedges_a\setminus\{(i^*_a,j^*_a)\},\weights)
\leq\bott(\hat{\Pi},\subedges_a\setminus\{(i^*_a,j^*_a)\},\weights)$.  
\end{proof}

\section{Proof of Proposition~\ref{prop:UnassAgentSwitch}}\label{app:ProofPropUnassAgentSwitch}
\begin{proof}
For an arbitrary order, $a\in\{1,\dots,n\}$, and an arbitrary unassigned agent, $i'\in\agents\setminus\{i^*_1,\dots,i^*_n\}$, consider the 
assignment, $\Pi'=\{\pi'_{i,j}\in\{0,1\}\,|\,(i,j)\in\agents\times\tasks\}$, obtained from $\Pi^*=\{\pi^*_{i,j}\in\{0,1\}\,|\,(i,j)\in\agents\times\tasks\}$ by switching the agent assigned to task $j^*_a$ from $i^*_a$ to $i'$, i.e.,
\begin{align*}
    \pi'_{i,j} =  
    \begin{cases}
        0 & \textrm{if }(i,j)=(i^*_a,j^*_a)\,,\\
        1 & \textrm{if }(i,j)=(i',j^*_b)\,,\\
        \pi^*_{i,j}& \textrm{otherwise}\,.
    \end{cases}
\end{align*}
We note that $(i',j^*_a)\in\subedges_a$ by definition in \eqref{eq:subgraph_k}. Because $(i',j^*_a)$ is not assigned according $\Pi^*$, it follows from Definition~\ref{def:lexass} that $w_{i',j^*_a}\geq\bott(\Pi^*,\subedges_a,\weights)$. By construction of $\Pi'$, we know that $\bott(\Pi',\edges',\weights)=\bott(\Pi^*,\edges',\weights)\leq\bott(\Pi^*,\subedges_a,\weights)$, with $\edges':=\subedges_a\setminus\{(i',j^*_a),(i^*_a,j^*_a)\}$. It follows that $\bott(\Pi',\subedges_a,\weights)=w_{i',j^*_a}$. 
Since $(i^*_a,j^*_a)$ is not assigned according to $\Pi'$, we know that $\bott(\Pi',\subedges_a\setminus\{(i^*_a,j^*_a)\},\weights)=w_{i',j^*_a}$. Finally, from \cref{def:bottass,,def:criticaledge} we have that $w_{i^*_a,j^*_a}+\mu_a=\Bop(\subedges_a\setminus\{(i^*_a,j^*_a)\},\weights)\leq\bott(\Pi',\subedges_a\setminus\{(i^*_a,j^*_a)\},\weights)$. 
\end{proof}

\bibliographystyle{IEEEtran}
\bibliography{IEEEabrv,CollisionAvoidancePaper}

\end{document}